\documentclass[10pt]{amsart}
\usepackage{latexsym,amsmath,amssymb,graphicx,epsfig,color,hyperref,supertabular,cite, empheq}
\usepackage{pdfpages}
\usepackage{pst-node}
\usepackage{tikz-cd} 
\usepackage[all,cmtip]{xy}
\usepackage[british]{babel}
\usepackage[T1]{fontenc}
\usepackage{enumitem}
\usepackage{amsmath,tikz}
\def\pplogo{\vbox{\kern-\headheight\kern -15pt
		\halign{##&##\hfil\cr&{%\sc
				\ppnumber}\cr\rule{0pt}{2.5ex}&\ppdate\cr} }} \makeatletter
\def\ps@firstpage{\ps@empty \def\@oddhead{\hss\pplogo}%
	\let\@evenhead\@oddhead % in case an article starts on a left-hand pagedix
	
}
\renewcommand{\thetable}{\@Alph\c@table}
\makeatother

\newtheorem*{theorem*}{Theorem}

\DeclareFontFamily{OT1}{rsfs10}{}
\DeclareFontShape{OT1}{rsfs10}{m}{n}{ <-> rsfs10 }{}
\DeclareMathAlphabet{\mathscript}{OT1}{rsfs10}{m}{n}

\suppressfloats
\def\ppnumber{\vbox{\baselineskip16pt }}
\def\ppdate{}
\date{} %necessary, to disable automatic dating by LaTeX

\newsavebox{\fmbox}

\author{Antonella Grassi}

\title[Bounds]{Bounds}

\address{
	Dipartimento di Matematica, Universit\`a di Bologna, Italy}
 \email{antonella.grassi3@unibo.it}
 \address{INFN, Sezione di Bologna,  Italy}
\address{
 Department of Mathematics, University of Pennsylvania, Philadelphia, PA, USA}
 \email{grassi@math.upenn.edu}
\theoremstyle{plain}
\newtheorem{thm}{Theorem}
\newtheorem{result}[thm]{Result (Physics)}

\newtheorem{corollary}[thm]{Corollary}

\newtheorem{proposition}[thm]{Proposition}

\newtheorem{conjecture}[thm]{Conjecture}
\newtheorem{expectation}[thm]{Expectation (Physics)}
\newtheorem{hc}[thm]{Charged Hypermultiplets}

\theoremstyle{definition}
\newtheorem{definition}[thm]{Definition}
\newtheorem{ex}[thm]{Example}

\theoremstyle{remark}
\newtheorem{remark}[thm]{Remark}

\title[Spectrum Bounds in Geometry]{Spectrum Bounds in geometry}

\numberwithin{thm}{section}

\begin{document}
	\dedicatory{To V. V. Shokurov for his 70+2 birthday}

%	\date{\today}
	
%	\ppnumber{\qquad \quad \quad \quad \quad \qquad \quad \quad \quad \quad\qquad \quad \quad \quad \quad  \qquad \quad \quad \quad \quad \qquad \quad \quad \quad \quad \quad CERN-TH-2018-013}
	
	%\subjclass[2000]{14C22, 14J70, 14M25, 32S35}
	\date{\today}

	%\subjclass[2000]{14C22, 14J70, 14M25, 32S35}
	%\subjclass[2000]{14C22, 14J70, 14M25, 32S35}
	\begin{abstract} 
Filipazzi, Hacon and Svaldi  proved that there are only finitely many  topological types of elliptically fibered Calabi-Yau threefolds.  We explore the implications of their results on the boundedness of the geometric quantities in the massless spectrum 
%An important question coming from string theory is the boundedness of the massless spectrum
of  
the F-theory Calabi-Yau compactifications. A key ingredient is what we call the geometric anomaly equation, and extension of the gravitational anomaly cancellation in physics, also to singular spaces.  We review and extend the dictionary between geometry and physics. We conclude with  explicit bounds.
	\end{abstract}

	\maketitle

\section{Introduction}\label{sec:Intro}
Recently Filipazzi, Hacon and Svaldi \cite{Filipazzi:2021dcw} have shown that there are only  finitely many types of  topologically different elliptically fibered Calabi-Yau threefolds, answering a long standing question from physics.  We explore the implications of their results and the previous work of  Gross \cite{Gross1994} as well as \cite{Grassi1991} on the boundedness of the geometric quantities in the massless spectrum 
%An important question coming from string theory is the boundedness of the massless spectrum
 of  
the F-theory Calabi-Yau compactifications.  The results build also on the work of Prokhorov and Shokurov \cite{PS09}.

An F-theory compactification  on a  Calabi-Yau variety $X$ is  an elliptic fibration  $X \to B$.  From now on, the threefolds are intended to be complex projective Calabi-Yau varieties. An important question coming from string theory is the boundedness of the massless spectrum of  the F-theory %Calabi-Yau 
compactifications, also in the context of the ``Swampland Program" \cite{Vafa:2005ui}.
The geometry of the Calabi-Yau, the structure of the fibration and the  massless particle spectrum in physics are closely related.
% and the relations that 
 %dictionary for the correspondence ????  between physics and geometry,  for fibrations between manifolds, 
In  Result   from Physics  \ref{thm:SpectrumP},  (1)-(4)   we review the descriptions of the spectrum for F-theory compactifications.  The dictionary between the quantities $(1$)-$(2$)-$(3)$ %-$(3')$-$(5)$ 
    in the physics spectrum and the geometry 
 %and their boundedness 
  is  immediate. The geometric counterparts of the quantity  (4),  the charged hypermultiples $H_{ch}$,  are discussed in  Section  \ref{sub:Hch}. In the same section we also outline  how to calculate explicitly some of the other geometric quantities in the spectrum,  with some relevant examples  in the Appendix  \ref{subsec:examples}.
 We extend the dictionary  first to  Calabi-Yau varieties with $\mathbb{Q}$-factorial terminal singularities with Definition \ref{def:specSing}--(3')  and then to  singular bases $B$ %fibrations
  with isolated multiple fibers  over the singularities of  $B$, by adding  (5) to the  (modified) spectrum  in Definition \ref{def:spectrum:quotients} ; see also \cite{ArrasGrassiWeigand}, \cite{GrassiWeigand2}. %We outline the definition of  the geometric quantities in the spectrum.
%We show that correspondences between  physcis and geometry 
These additions to the massless spectrum are consistent with the findings in physics; we show also that they are natural from the geometry of the Calabi-Yau. 
In Corollary \ref{cor:bound12smooth}, Corollary \ref{cor:boundsHsing} and Corollary \ref{cor:boundsIII}  %and Corollary \ref{cor:boundsIV} %building on the work of  \cite{Gross1994} and \cite{Filipazzi:2021dcw}
 we  show the  boundedness of the quantities in  ($1)$-$(2)$-$(3)$-$(3')$-$(5)$  in the spectrum.  
 
 To prove the boundedness of the quantity (4) in  the spectrum \ref{thm:SpectrumP}, %Then boundedness  of 
 the charged multiples $H_{ch}$,   we use  the Gravity Anomaly Cancellations and its extensions (Corollaries \ref{cor:boundIan}, \ref{cor:boundsan2} and \ref{cor:boundsIIIanom}).  (In the physics theory the Gravitational Anomaly Cancellation Formulas %in Physics
  must in fact  be satisfied  by  the quantities in the spectrum for the theory to be consistent.)
 % Geometric Anomaly Cancellation, %the quantity (4) in  the spectrum \ref{thm:SpectrumP}, follows from
 % the previous Corollaries,  the  Gravity Anomaly Cancellations and its extensions (Corollary \ref{cor:boundIan}, \ref{cor:boundsan2} and \ref{cor:boundsIIIanom}.)
%In fact, for the theory to be consistent the quantities in the spectrum  must  satisfy 
%The Gravitational Anomaly Cancellation Formulas in Physics must in fact  be satisfied  by  the quantities in the spectrum for the theory to be consistent.
%We review the Gravitational Anomaly Cancellation Formula in Section \ref{FtheoryAnomalies}:
However, the Gravitational Anomaly Cancellation \ref{prop-anomalies-gen} was  originally derived in physics   for fibrations with section between manifolds. We extend it in   Proposition \ref{prop:anomequ} and Theorem \ref{prop:anomequm} to  Calabi-Yau  with $\mathbb{Q}$-factorial terminal singularities and then to equidimensional  fibrations with isolated multiple fibers; see also \cite{GrassiWeigand2}.
%Building on the works of  \cite{Gross1994} and \cite{Filipazzi:2021dcw} we prove that  
% \cite{GrassiWeigand2} and \cite{Lee:2019skh} together with
%the anomaly cancellation formula imply boundedness of the massless spectrum for elliptic Calabi-Yau threefolds. 

%The dictionary between the quantities $(1$)-$(2$)-$(3)$-$(3')$-$(5)$   in the physics spectrum and the geometry %and their boundedness 
% is  immediate. The geometric counterparts of the quantity  (4), the charged hypermultiples $H_{ch}$,  are discussed in  Section  \ref{sub:Hch}. In the same Section we also outline  how to calculate explicitly some of the other geometric quantities in the spectrum,  with some relevant examples  in the Appendix  \ref{subsec:examples}.
 In Corollary \ref{cor:boundsIV}  and  Corollary \ref{cor:boundsV}   we apply the boundedness of the spectrum   and a geometric formulation of the anomaly cancellation to deduce boundedness of the type of singular fibers (and of the matter representations) and the components of the discriminant.
We conclude with % an implication of the bounds of discrete symmetries and by 
%stating 
explicit bounds %from physics
in Section \ref{subsec:ExplBounds}.
%The geometry of a Calabi-Yau  is closely related to the  massless particle spectrum 

%Among the  geometric counterparts of the quantities in the massless spectrum are   quantities associated to the singular fibers of the fibrations and the rank of the Mordell-Weil group, if the fibration has a section \ref{thm:SpectrumP}. The geometry of a Calabi-Yau  is closely related to the  massless particle spectrum and the relations that  We review  the  dictionary for the correspondence and extend it to include the case of singular Calabi-Yau $X$ and isolated multiple fibers, over a singular basis $B$ \ref{prop:anomequm}. We outline how to calculate explicitly  the geometric quantities in the spectrum.

%sIf $X$ is a threefold and the fibration is  isotrivial  then  $\kod (B) < 0 $ and $B$ is rational, otherwise $B$ is an Enriques surface. %The isotrivial fibration will be considered in a separate work.

\section{Elliptic  Fibrations and F-theory}\label{sec:elliptic}
All the varieties are assumed to be complex projective algebraic.
\begin{definition} A   Calabi-Yau  threefold $X$ is a variety 
	% if it is a complex normal projective  threefold
	 with $\mathbb{Q}$-factorial terminal singularities,  $h^i(X, {\mathcal O}_X)=0 $, $  i  = 1,2, $ and $K_X \sim {0}$.
\end{definition}
\begin{definition} An elliptic   fibration is a morphism $\pi : X\rightarrow B$ whose fibers over a dense set in $B$ are elliptic 
	curves. The complement of this dense locus in $B$, the discriminant of the fibration, is  denoted by  $\Sigma_{X/B }$.
Let 	$\{\Sigma_j \}$ be the irreducible codimension one components of  $\Sigma_{X/B }$.
%	When a genus one fibration  $\pi$ has a section $\sigma: B \to X$, we refer to as an elliptic fibration. 
Let  $\Sigma_m\stackrel{def}= \ \{ Q \in B \text{ such that }   \pi   ^{-1} (Q)  \ \text{ is a multiple fiber}\} $
 
\end{definition}
%$\Sigma_{X/B}\subset B$ 
%The support of the discriminant locus of a $X \to B$ a genus one fibration 
%has  a stratified structure given by its singularities.

\begin{thm}\label{thm:birminmodel}
	Let $Y$ be a birational Calabi-Yau threefold  and $%\pi: 
	Y \rightarrow S$ be an elliptic fibration. Then
	there exists a birational equivalent elliptic  fibration $  \pi :  X  \to B $ such that $ X $ is minimal, with  $\mathbb{Q}$-factorial terminal singularities and 
	\begin{enumerate}
		%	\item $\k(X)= \k (B, \Lambda)$.
		%	\item If $K_B+ \Lambda$ is not pseudo-effective, there exists a birational equivalent fibration $ \pi :  X  \to B $, $ X $ with $\mathbb{Q}$-factorial  terminal singularities, $(, \bar \Lambda)$ with  klt singularities such that $K_{ X }  \equiv \bar{\pi}^*(K_{ B } + \bar \Lambda)$. \\
		%	$X$ is birationally a Mori fiber space.
		
		\item 
		%If $K_B+ \Lambda$ is pseudo-effective, equivalently $\kappa (X) \geq 0$, and klt flips exist and terminate in dimension $n-1$,
		%	\begin{enumerate}
			%	\item 
			$( B ,   \Delta )$ has $\mathbb{Q}$-factorial klt singularities,  $K_{ X }  \equiv  {\pi}^*(K_{ B } +   \Delta )$ and 
			$  \Delta  $ is supported on  the discriminant  of the fibration.
			\item   $ \pi :  X  \to B $
			% in either (2) or (3) above
		is 
			equidimensional.
			
% \item If  there are no multiple fibers, $ B $ is smooth.
			
\item  The locus $\Sigma_m$ %\stackrel{def}= \ \{ Q \in B \text{ such that } (  \pi  ) ^{(-1)} (Q)  \ \text{ is a multiple fiber}\} $
 consists of a set of points.

\item If $ B $ is singular at $Q$  then $Q \in \Sigma_m$ and the singularity is  a rational double point of type $A_n$.

\item If $  \Delta  \neq 0$, $ B $ is rational.

\item  If $  \Delta  = 0$ the singularities of $ B $  are rational double points, the minimal resolution of $ B $ is an Enriques surface,  the fibration $  \pi  $ is isotrivial.% and the only singular fibers are multiples of a smooth elliptic curve.
			
			%If furthermore we have that $\kappa(K_{X_0}) \geq 0$, we can take $Y$ to be a minimal model.
			%With further birational modifications, we can obtain a birational model $ X  \rightarrow \bar{B}$ satisfying the above with the additional property that it factors into $			 X  \rightarrow Z \rightarrow \bar{B}$ such that $Z \rightarrow \bar{B}$ is a birational morphism and $ X  \rightarrow Z$ is a morphism that is equidimensional over an 			open $U \subset Z$ where $\codim(Z - U) \geq 3$
		\end{enumerate}
	\end{thm}
\begin{proof}
(1)-(2) and (5)-(6)  are proved in previous works of the author \cite{Grassi1991}, \cite{GrassiSingBase1993} and  \cite{GrassiEqui}.  (3) follows from  Kodaira's canonical divisor formula for surfaces and $K_{ X } \simeq {\mathcal O}_{ X }$.
%(4) also shows that if $P$ is in the smooth locus of $B$, then $ \pi  ^{-1} (P) $  is not a multiple fiber. 
(4) follows form \cite{GrassiSingBase1993}, \cite{Grassi1991} and  \cite{Gross1994}.
%Assume that $P$ is in the singular locus of $B$, but  $ \pi ^{-1} (P) $  is not a multiple fiber. Then by (3), either there are birationally equivalent models with a smooth base or $  \pi  $ must have singular fibers that is  $ \Sigma_m \neq \emptyset$. Let us consider the Jacobian fibration $J \to B$. We can assume that $J$ is a Calabi-Yau, and by (3)  there exists $J' \to B'$, with $B'$ non singular, $ \phi:  B  \to B'$ birational, and $\phi_{| B  \setminus \Sigma_m}$ is an isomorphism. 
\end{proof}
%In most examples %the  multiple fibers  occur over singular points of $B$ which 
%the points $Q$   in the set $\Sigma_m$ of (4)  are disjoint from  the divisorial component of $\Sigma$ and $B$ is singular at %$B$. 
%In which case we have:
%\begin{proposition}\cite{GrassiWeigand2}	Let $\pi: X \to B$ an elliptic Calabi-Yau fibration as in Theorem \ref{thm:birminmodel}.
%	Assume that $P \in \Sigma$ but $P$ is disjoint from the divisorial component of $\Sigma$. Then $P \in \Sigma_m$ and either $B$ singular at $P$ and/or $X$ has a singularity at a point in $\pi^{-1}(P)$.\\
%	\bla{The statement might only be needed for Stringy Kodaira in the other paper.}
	
%	\todo{ There is a claim before Theorem 2 in GrassiWen, attributed to Kawamata\cite{Kawamata83Certain}: X smooth, B smooth implies $\Sigma$ divisor. I cannot find the reference, but I can prove it for $X$ with $K_X \simeq {\mathcal O}_X$.} 
%\end{proposition}

The converse of the statement (4) does not hold: The analysis of the Example B  in Appendix B of \cite{Oehlmann:2019ohh} gives a counterexample, namely there is an elliptic fibration $\pi:  X \to B$, $X$ and  $B$ smooth, $Q \in \Sigma_m$ and in a divisorial component of $\Sigma$.% \cite[Appendix B]{Oehlmann:2019ohh}.

\subsection{F-theory, Massless spectrum,  Gravitational Anomaly, for manifolds}\label{FtheoryAnomalies}~

\smallskip

F-theory compactifications were first constructed on elliptic Calabi-Yau manifolds with a section\footnote{In the physics literature an  elliptic fibration is assumed to be an elliptic fibration with a section, otherwise %the fibrations are 
	it called a genus one fibration. }  \cite{Vafa:1996xn, Morrison:1996na, Morrison:1996pp}. When there is a section there are no multiple fibers and the base of the fibration $B$ can be assumed to be birationally  smooth \ref{thm:birminmodel}. The physics theory with  the presence of multiple fibers is still being understood. We discuss the multiple fibers in %the next 
Section \ref{subsec:multf}.

\begin{result}[Gauge algebra] Let $X \to B$ be an equidimensional elliptic fibration  with $X$ a  smooth  Calabi-Yau threefold. % and $B$ smooth. 
	The gauge algebra of the F-theory compactified to $ X\to B$   is a naturally associated 
	algebra   % the gauge algebra
	$\mathfrak g  = \bigoplus_{j} \mathfrak{g}(\Sigma_j) \oplus \mathfrak{u}(1)^{\oplus r}$
	where
	\begin{itemize}
				\item $\mathfrak{g}(\Sigma_j)$ are  either  semisimple Lie algebras or  "twisted  algebras", which are quotients of  the extended (affine) semisimple Lie algebras,
		\item 
		the sum is taken over the irreducible divisorial components $\{\Sigma_j\} $ of the discriminant locus, %and 
		\item $r$ is the rank of the  Mordell-Weil group of sections of the elliptic fibration. 	\end{itemize}

	\end{result}

 $\mathfrak{u}(1)^{\oplus r}$ is called  the  abelian component and $\bigoplus_j \mathfrak{g}(\Sigma_j) $ the non-abelian component of $ \mathfrak{g}$. We write ${\mathfrak g}_j\stackrel{def}=\mathfrak{g}(\Sigma_j)$.

\begin{remark}\textbf{(Geometry)}
There are several proposed  methods to determine the algebras $\mathfrak{g}(\Sigma_j) $ with mathematical constructions.  If the fiber over the general point of $\Sigma_j$ is a cusp or a node ${\mathfrak g}(\Sigma_j)= \{e\}$.
If the fibration has a section $\mathfrak{g}(\Sigma_j) $ is either the simply laced Lie algebra associated to the  Kodaira fiber over the general point of $\Sigma_j$, or  the non-simply laced ones,   a quotient by an outer automorphism of the simply laced ones if  the rational curves in the general fiber of $\Sigma_j$  have a relation  in the Mori cone $\operatorname{NE}(X)$ of effective curves.  If the fibration does not have a section,  \textit{twisted} algebras might also   occur   in the presence of multiple fibers, again when   the rational curves in the general fiber of $\Sigma_j$  are dependent   in %the Mori cone
$\operatorname{NE}(X)$. The twisted algebra are  then counted with multiplicity, \cite{BraunMorrison} and \cite{Kohl:2021rxy}.
 We present  and review different constructions and properties in \cite{GrassiWeigand2}. 
\end{remark}

  By duality, the construction implies that the following  always hold:
\begin{result}
$\sum_j \operatorname{rk} \mathfrak{g}(\Sigma_j) < \operatorname{rk} (\operatorname{Pic}(X)) %= h^{1,1}(X)
$.
\end{result}
\begin{result}[Gauge algebra and spectrum]\label{thm:SpectrumP} Let $X \to B$ be an equidimensional elliptic fibration  with $X$ a smooth  Calabi-Yau threefold and $B$ smooth. % and rational. 
	% The (abelian) gauge group $G$ 
	%	gauge algebra $\mathfrak{g}$ of F-theory compactified on $X$ defined above is $G= \prod_{a=1}^{r} U(1)_a$, where $r$ is the rank of the Mordell-Weil group $MW(Y/B)$.
	The massless physical spectrum of the 
	%(abelian) gauge group $G$ 
	gauge algebra $\mathfrak{\mathfrak{g}}$ of the F-theory compactified on $X$ 
consists of 
	\begin{enumerate}
		\item  $V=h^{1,1}(X)-h^{1,1}(B)-1 + \dim {\mathfrak g} - \operatorname{rk} {\mathfrak g}$ vector multiplets,
		\item $T= h^{1,1}(B)-1 $ tensor multiplets,
		\item $H = H_{unch} + H_{ch}$ hypermultiplets, where  $H_{unch} = h^{2,1}(X) +1= \frac{1}{2} b_3(X)$  is the number of uncharged multiplets and $H_{ch}$ the number of hypermultiplets charged under $\mathfrak{g}$, %The  hypermultiplets $H_{ch}$ are  the dimension of the charged representations, with multiplicity, defined  in the Assignments  \ref{agas}, \ref{unloc-rep} and \ref{loc-repP}, \ref{Method:u1hypersP}.
		\item  one universal gravity multiplet.
	\end{enumerate}  

\end{result}

\begin{remark}\textbf{(Geometry)}  After ${\mathfrak g}$ is determined, the quantities in the massless spectrum have immediate  geometric counterparts, with the exception of the charged multiplets $H_{ch}$.
They
%  hypermultiplets $H_{ch}$
 are 
 % (multi)-representations 
 charged by  the gauge algebra ${\mathfrak g}$.  In the geometry  of $F$-theory compactifications, the  hypermultiplets $H_{ch}$ are  associated to the singular fibers of the fibration (see Section \ref{sub:Hch} and \cite{GrassiWeigand2}).
The general definition and construction of $H_{ch}$ in physics and in mathematics %al language 
is still largely an open question. We refer  for example  to \cite{GrassiWeigand2} \cite{Weigand:2018rez}  for a review of different mathematical constructions. We summarize the main points in Section \ref{sub:Hch}.
 \end{remark} 

	\begin{corollary}[Bounds, I]\label{cor:bound12smooth}  Let $X \to B$ be  an equidimensional elliptic fibration  with $X$ smooth  Calabi-Yau threefold and $B$ smooth.
	The quantities (1), (2) and $H_{unch}$ in (3)   of the massless spectrum are bounded.
\end{corollary}

\begin{proof}
By 	%Proposition ?? in 
 \cite{Filipazzi:2021dcw}, the equidimensional model $X \to B$ is isomorphic to a fiber in a bounded family $\mathcal X$;  $X \simeq X_t$,  and $B \simeq S_t$, for some $t$ in a bounded family $\mathbf{S}$.   $T$ is bounded and so is $V$, because the  $\operatorname{rk} (\operatorname{Pic}(X))$  is bounded \cite{Gross1994}. It follows also that  $H_{unch}$  is bounded.
\end{proof}

\smallskip

The 6-dimensional effective theory obtained by compactification of  F-theory on $X$  is consistent if the anomalies are satisfied. In particular 

\begin{result} [Gravitational Anomalies {\cite{Green:1984bx}}]\label{prop-anomalies-gen}
	Let $X \to B$ be an equidimensional elliptic fibration   {without multiple fibers}, with $X$  a smooth Calabi-Yau threefold, $B$ smooth  and rational. 	The Gravitational %, mixed gravitational$-U(1)_a-U(1)_b$ and abelian $U(1)_a - U(1)_b - U(1)_c - U(1)_d$ 
Anomalies are cancelled by the 6-dimensional Green-Schwarz mechanism if the following equation holds:
	% \cite{Park:2011ji}:
	\begin{eqnarray*}
		H -V +29T=273
%H_{unch} - h^{1,1}(X)-h^{1,1}(B)-1 + \dim {\mathfrak g} - \operatorname{rk} {\mathfrak g}+ 29 (h^{1,1}(T)-1)&=& 273  -H_{ch}
		%	(-K_B) \cdot b_{a,b} &=&   \frac{1}{6} \sum_I N_I q^I_a q^I_b  \label{grav-U1U1} \\
		%	b_{a,b} \cdot b_{c,d}  + b_{a,c} \cdot b_{d,b}     + b_{a,d} \cdot b_{c,b}     &=& \sum_I N_I q^I_a q^I_b q^I_c q^I_d  \,.\label{U14}
	\end{eqnarray*}
where $H, V$ and $T$ are as in \ref{thm:SpectrumP}.
	%	$ b_{a,b}$ on the   the lefthand side  of (\ref{grav-U1U1}) and (\ref{U14}) is defined in Section \ref{ShiodaHeightP}, Definition \ref{defShioda}.
\end{result}

%In addition to the $V=h^{1,1}(X_i)-h^{1,1}(B)-1 =  {\rm rk(MW}(X_i/B)) =10$ vector multiplets and the universal gravitational multiplet, the massless spectrum comprises $T= h^{1,1}(B)-1 = 9$ tensor multiplets as well as $H = H_{unch} + H_{ch}$ hypermultiplets.
%Here  $H_{unch} = h^{2,1}(X_i) +1 =4$, Theorem \ref{RossiNamikawa} counts the hypermultiplets uncharged under $G$ and $H_{ch}$ the number of hypermultiplets charged under $G$.

%{ (1), (2) and (4)  immediately provide a correspondence between the massless spectrum and the birational invariants of the elliptic Calabi-Yau $X$.
	%   In the physics literature mostly smooth models have been considered; in particular the ``F-theory" interpretation of the correpondence between singularities and algebras is on manifolds. However,  $\mathbb{Q}$-factorial terminal  singularities occur naturally, for example in the Jacobian variety of a genus one fibration  \cite{BraunMorrison}, in certain fiber products of rational elliptic surfaces  \cite{Morrison:2016lix},
	% as well as F-theory duals of generic non-geometric compactifications of the heterotic string as studied in \cite{LustTFects,Font:2017cya}. 
	
	% Theorem \ref{Anomalycondition1}  is consistent with the cancellation of Gravitational Anomalies in the six-dimensional effective theory obtained by compactification of ``F-theory" on $X$, 
	% \begin{equation} \label{HminusV}
		%	H - V + 29 T  = 273, 
		% \end{equation}
	% even when $X$ has singularities. Here
	% \begin{equation*}
		%	{\rm dim}(\mathfrak{g})=V, \ \ h^{1,1}(B) -1=T
		% \end{equation*}

Corollary \ref{cor:bound12smooth} then implies:
\begin{corollary}[Bounds, I']\label{cor:boundIan} 
		Let $X \to B$ be an equidimensional elliptic fibration   {without multiple fibers}, with $X$  a smooth Calabi-Yau threefold, $B$ smooth  and rational. Assume that the Gravitational Anomalies in  \ref{prop-anomalies-gen} are cancelled. Then $H_{ch}$  is bounded.
	\end{corollary}

%\section{Bounds, I}\label{sec:boundsI}

%	\begin{proposition}\label{pro:bound12smooth}  Let $X \to B$ an equidimensional elliptic fibration  with $X$ smooth  Calabi-Yau threefold and $B$ smooth.
%The quantities in the massless spectrum (1), (2) and $H_{unch}$ in (3)  are bounded.
%\end{proposition}

%s\begin{proof}
%$V$ is bounded by Pic, $T$  by \cite{Filipazzi:2021dcw},   $H_{unch}$ etc.
%\end{proof}

\section{Massless spectrum,  \texorpdfstring{$H_{ch}$}, The geometric anomaly equation and singularities}\label{sec:GeomAnom}

We turn now our attention to the geometric interpretation and the extension of the anomaly cancellation formula to Calabi-Yau varieties with $\mathbb{Q}$-factorial terminal singularities. The motivation is twofold: the first one is  to describe in geometric terms $H_{ch}$, the second to derive the (extended) anomaly cancellation formula  consistent for mathematics and physics when  $X$ and  %has $\mathbb{Q}$-factorial terminal singularities 
%and when 
$B$ are also singular, as in Theorem \ref{thm:birminmodel}.

% We will discuss the case of the singularity of $B$ in ???? separately.

  In  \cite{ArrasGrassiWeigand} and \cite{GrassiWeigand2} we show that (3) in  the 
 	massless spectrum \ref{thm:SpectrumP}
   must be modified %as follows  
   for Calabi-Yau with $\mathbb{Q}$-factorial terminal singularities.
%  In  \cite{ArrasGrassiWeigand} and \cite{GrassiWeigand2} we show that, building on \cite{NamikawaSteenbrink} (3) in  the 
  %	massless spectrum \ref{thm:SpectrumP}
%  must be modified as follows  for Calabi-Yau with $\mathbb{Q}$-factorial terminal singularities
Gorenstein terminal singularities are isolated hypersurface  singularities and we can define the Milnor number:
\begin{definition} Let
	$(\mathcal U, 0)\subset \mathbb{C}^{n+1}$ be a neighborhood of an isolated hypersurface singularity $P=0$,  defined by ${f=0}$. The Milnor number of $P$  can be defined as
	\begin{align*}
		m(P) =\dim_{\mathbb C}( \mathbb C\{ x_1, \ldots, x_{n+1} \}/{\scriptstyle< f, \frac{\partial f}{\partial x_1}, \ldots,   \frac{\partial f}{\partial x_{n+1}} } >).\,
	\end{align*}
\end{definition}
  Also: $ \mathrm{CxDef}(X) +1 = \frac{1}{2} (b_3(X)+\sum_P m(P)) $  \cite{NamikawaSteenbrink}. Then: 
  
  \begin{definition}\label{def:specSing}[Massless Spectrum with $\mathbb{Q}$-factorial terminal singularities]	Let $X \to B$  be an equidimensional elliptic fibration    {without multiple fibers}, with $X$ a Calabi-Yau threefold with $\mathbb{Q}$-factorial terminal singularities and $B$ smooth. 
   (3) in  Massless Spectrum  \ref{thm:SpectrumP}  is replaced by 
\begin{eqnarray*}
(3') \quad 	H_{unch} \stackrel{def}=\mathrm{CxDef}(X) +1 = \frac{1}{2} (b_3(X)+\sum_P m(P)) ,
\end{eqnarray*}
where $m(P)$ is the Milnor number of the hypersurface singularity at $ \in X$.
\end{definition}

The extended version of Corollary \ref{cor:bound12smooth} holds:
 \begin{corollary}[Bounds,  II]\label{cor:boundsHsing}
	Let $X \to B$ be an equidimensional elliptic fibration    {without multiple fibers}, with $X$ a Calabi-Yau threefold with $\mathbb{Q}$-factorial terminal singularities and $B$ smooth. 
	The quantities (1), (2) and %$H_{unch}$  as 
	$(3') $ 
	% \ref{def:specSing} 
	 of the  extended massless spectrum are bounded. 
 \end{corollary}
\begin{proof}
 The Milnor number is a topological invariant of the isolated singularities of the minimal Calabi-Yau threefolds and we conclude by \cite{Filipazzi:2021dcw}.
\end{proof}
We write the   results of \cite{ArrasGrassiWeigand} and \cite{GrassiWeigand2} as:
\begin{proposition}\label{prop:anomSingX}
	Let $X \to B$ an equidimensional elliptic fibration    {without multiple fibers}, with $X$ a Calabi-Yau threefold with $\mathbb{Q}$-factorial terminal singularities and $B$ smooth.   
	Let the massless spectrum be as in Definition \ref{def:specSing}. 
	
The 	Anomaly Cancellation with the modification in Definition \ref{def:specSing} remains
	formally unchanged as  \begin{equation*} \label{HminusV}
		H - V + 29 T  = 273.
	\end{equation*}
\end{proposition}
\begin{proof}
The dimension of the complex deformations of $X$, $\mathbb{Q}$-factorial Calabi-Yau is  computed by  $\mathrm{CxDef}(X)=\frac{1}{2}(b_3(X) - 1 +\sum_m m(P))$ \cite{NamikawaSteenbrink}.
\end{proof} 

\begin{remark}
	In  \cite{ArrasGrassiWeigand} we also  explicitly check the consistency in physics of the Gravitational Anomaly Cancellation with the modified definition \ref{def:specSing}.
	
\end{remark}

\begin{proposition}\label{prop:anomequ}
	Let $X \to B$ be an equidimensional elliptic fibration    {without multiple fibers}, with $X$ a Calabi-Yau threefold with $\mathbb{Q}$-factorial terminal singularities and $B$ smooth.   Let $m(P)$ %and $\tau(P)$
	be  the Milnor  number of  the point $P \in X$.  % Let	$H^1(X, T_X)$ and  $\delta$ be as in equation  \ref{eq:LocGlob}. \\
	Then the Anomaly Cancellation  Equation \ref{prop-anomalies-gen}  with the modification of Definition \ref{def:specSing} becomes:
	\begin{eqnarray*}
		30 K_B^2+   \frac{1}{2}{  (\chi_{top}(X)-\sum_P m(P))}  &=	& 	H_{ch}  -	(\dim {\mathfrak g} - \operatorname{rk} {\mathfrak g} ).
	\end{eqnarray*}

\end{proposition}
\begin{proof}  We extend the arguments of \cite{GrassiMorrison11} and \cite{GrassiWeigand2}.
	%	\begin{eqnarray*}
	%	H -V +29T &=&273\\
	%	H_{unch} - h^{1,1}(X)+h^{1,1}(B)+1+ 29 (h^{1,1}(T)-1)&=& 273  -H_{ch}  + \dim {\mathfrak g} - \operatorname{rk} {\mathfrak g}\\
	%	273  -  \frac{1}{2}{  \sum_P m(P)}	&=& H_{ch}  - ( \dim {\mathfrak g} - \operatorname{rk} {\mathfrak g}) + 29 T 
		%	(-K_B) \cdot b_{a,b} &=&   \frac{1}{6} \sum_I N_I q^I_a q^I_b  \label{grav-U1U1} \\
		%	b_{a,b} \cdot b_{c,d}  + b_{a,c} \cdot b_{d,b}     + b_{a,d} \cdot b_{c,b}     &=& \sum_I N_I q^I_a q^I_b q^I_c q^I_d  \,.\label{U14}
%	\end{eqnarray*}
Note that by hypothesis $B$ can assumed to be smooth.  
%	$h^{1,1}= ...$. 
%	Following \cite{ArrasGrassiWeigand} and \cite{GrassiWeigand2}   we modify $H_{unch}= h^{2,1}(X)$  to the dimension of the complex deformations of $X$, which  for a $\mathbb{Q}$-factorial Calabi-Yau is  $\frac{1}{2}(b_3(X) + 1 +\sum_m m(P))$.
The argument  in  \cite[Section 6]{GrassiMorrison11} for $\operatorname{rk}(MW(X/B)) \neq \{ e\}$ and $B$ rational builds on  the relations (1), (2), (3) of the Physics Result \ref{prop-anomalies-gen} and on Noether's formula for smooth surfaces. The same  relations provide the result  when  $\operatorname{rk}(MW(X/B)) \neq \{ e\}$   and when there is no section,  with  Shioda-Tate-Wazir formula as  modified by \cite{BraunMorrison} for fibration without a section.
	  The statement then follows from  (5) and (6) in  Theorem \ref{thm:birminmodel}. 
\end{proof}  
	  
\begin{remark}
The statement also  includes the case of when $B$ is  an Enriques surface.
\end{remark}
\smallskip

\begin{corollary}[Bounds, II']\label{cor:boundsan2}
 Let $X \to B$ be an equidimensional elliptic fibration, with $X$ a Calabi-Yau threefold with $\mathbb{Q}$-factorial terminal singularities. Then the Gravitational Anomaly Cancellation implies that $H_{ch}$ is bounded.
\end{corollary}

\subsection{Multiple fibers}\label{subsec:multf}~

\smallskip

The interpretation of the  physics of elliptic fibrations in the presence of multiple fibers is still  a topic 
in its initial stages and much  %of 
investigation is needed. However  we have  the following:
%there are several results  in this direction \cite{Anderson:2018heq}  confirmed in \cite{Anderson:2019kmx}, \cite{Kohl:2021rxy}:

\begin{result}[\cite{Anderson:2018heq},  \cite{Anderson:2019kmx}, \cite{Kohl:2021rxy}]\label{anom:quotients} Let $X \to B$ be an equidimensional elliptic fibration, %with 
	$X$ a Calabi-Yau threefold with $\mathbb{Q}$-factorial terminal singularities, $\Sigma_m$ disjoint from the divisorial components of the discriminant and $B$ possibly singular.  % Assume that $\Sigma_m$ is disjoint from the divisorial components of $\Sigma$ that is $\Sigma_m \not \subset \Sigma_j,  \ \forall j$.
The  spectrum in \ref{def:specSing} and the Anomaly Cancellation in Results \ref{prop-anomalies-gen} must be modified for the theory to be consistent. 
\end{result}

The following Definitions  \ref{def:spectrum:quotients}  and  \ref{def:anom:quotients} are motivated by the findings for  Calabi-Yau quotients and Calabi Yau symmetric toroidal  orbifolds in \cite{Anderson:2018heq} \cite{Anderson:2019kmx} and \cite{Kohl:2021rxy}.  One example of such fibration  is $X$, a  quotient of an elliptic Calabi-Yau  threefold $Y$ by $\mathbb{Z}/{(m+1)}\mathbb{Z}$ such that the group acts without fixed points and it preserves the fibration.  $X$ is smooth  with  isolated multiple fibers over the singular points, of type $A_{m}$, of the quotient base. 
%The singularities of $B$ are at worse  $A_{m_i}$ singularities, because the fibration $X\to B$ is  also equidimensional. Then: 
% Assume that $\Sigma_m$ is disjoint from the divisorial components of $\Sigma$ that is $\Sigma_m \not \subset \Sigma_j,  \ \forall j$. $Q _i\in \Sigma_m \in B$ is a singular point, necessarily of type $A_{m_i}$.
%\end{lemma}
In fact
\begin{remark} If $X \to B$ is any  equidimensional elliptic fibration, %with 
	$X$ a Calabi-Yau threefold with $\mathbb{Q}$-factorial terminal singularities, then the  singularities of $B$ are at worse  of type $A_{m_i}$, because   there exists an equidimensional  birational model $\bar X \to \bar B$ such that $\bar B$ has at worse  $A_{m_i}$ singularities (Theorem \ref{thm:birminmodel}).
\end{remark}

%\begin{result}\label{spectrum:quotients} Let $X \to B$ an equidimensional elliptic fibration, with $X$ a Calabi-Yau threefold with $\mathbb{Q}$-factorial terminal singularities and $B$ possibly singular.  Assume that $\Sigma_m$ is disjoint from the divisorial components of $\Sigma$ that is $\Sigma_m \not \subset \Sigma_j,  \ \forall j$.
%	The  Spectrum in \ref{def:specSing}  must also include
%	\begin{equation*}
%	(5) \mathcal{A}_n= H 
%	\end{equation*}
%\end{result}
 We can give the following:
\begin{definition}\label{def:spectrum:quotients}[Massless Spectrum, with singularities and  multiple fibers]
Let $X \to B$ be an equidimensional elliptic fibration, with $X$ a Calabi-Yau threefold with $\mathbb{Q}$-factorial terminal singularities. % and $B$ possibly singular. 
 Assume that $\Sigma_m$ is disjoint from the divisorial components of $\Sigma$, that is $\Sigma_m \not \subset \Sigma_j,  \ \forall j$.  If $Q _i\in \Sigma_m  \subset  B$ is a singular point, %, necessarily of type $A_{m_i}$.% by Theorem \ref{thm:birminmodel}.
  %Let $n_i$ be the number of such points.  
 % T
  then the spectrum in \ref{def:specSing}  must also include for each such $Q_i$, necessarily of type $A_{m_i}$,

\begin{enumerate}

	\item [(5.a)] One  neutral hypermultiplet
	\item  [(5.b)] $  (m_i  )$ tensor multiplets
\end{enumerate}

\end{definition}

\begin{definition}\label{def:anom:quotients}[Anomalies, with singularities and  multiple fibers]
 In the same hypothesis of Definition \ref{def:spectrum:quotients}, let $n_i$ be the number of such points.  
The  Anomaly Cancellation in Results \ref{prop-anomalies-gen} must be modified as follows for the theory to be consistent:
\begin{equation*} \label{HminusVm}
	H - V + 29 T + \sum_{i} n_i \cdot  (m_i)  = 273, 
\end{equation*}

\end{definition}

%\begin{remark}
%Definitions  \ref{def:spectrum:quotients}  and  \ref{def:anom:quotients} are motivated by the findings for  Calabi-Yau quotients and Calabi Yau symmetric toroidal  orbifolds in \cite{Anderson:2018heq} \cite{Anderson:2019kmx} and \cite{Kohl:2021rxy}.  One example of such fibration  is a  quotient of an elliptic Calabi-Yau threefold by $\mathbb{Z}/n\mathbb{Z}$ which acts without fixed points on $X$ in order to have isolated multiple fibers over the singular points of the quotient of the base.
%\end{remark}
% We do not report the modified statement because it requires additional notations and definitions and because in \cite{GrassiWeigand2}  we prove that

The Geometric Anomaly Cancellation Formula remains however unchanged:

\begin{thm}[\cite{GrassiWeigand2}]\label{prop:anomequm}
	Let $X \to B$ be an equidimensional elliptic fibration, with $X$ a Calabi-Yau threefold with $\mathbb{Q}$-factorial terminal singularities and $B$ possibly singular.  Assume that $\Sigma_m$ is disjoint from the divisorial components of $\Sigma$, that is $\Sigma_m \not \subset \Sigma_j,  \ \forall j$.  Let $m(P)$ %and $\tau(P)$
	be  the Milnor  number of  the point $P \in X$.  % Let	$H^1(X, T_X)$ and  $\delta$ be as in equation  \ref{eq:LocGlob}. \\
	Then the  modified Anomaly Cancellation  Equation \ref{prop-anomalies-gen}  and \ref{anom:quotients} with the modification of Definition \ref{def:specSing} remains:
	\begin{eqnarray*}
		30 K_B^2+   \frac{1}{2}{  (\chi_{top}(X)-\sum_P m(P))}  &=	& 	H_{ch}  -	(\dim {\mathfrak g} - \operatorname{rk} {\mathfrak g} ).
	\end{eqnarray*}

\end{thm}

Note that $K_B^2$ is an integer because the singularities are at worse rational double points. 
% The equation in  \ref{prop:anomequm} must however be modified when $X$ has $\mathbb{Q}$-factorial terminal singularities, as  observed first in \cite{BraunMorrison}.
% Because $h^{2,1}(X)$ is the dimension of the complex deformation for a smooth Calabi-Yau threefold  and because the Gorenstein terminal singularities are hypersurface  singularities in \cite{ArrasGrassiWeigand} and \cite{GrassiWeigand2} we make the following:å
%\begin{definitiont}\label{prop:equiv}
%Let   $X \to B$ be a genus one Calabi-Yau threefold with $\mathbb{Q}$-factorial  terminal singularities as in %Theorem \ref{thm:birminmodel} and $B$ smooth.    Let $m(P)$ %and $\tau(P)$
%	be  the Milnor  number of  the point $P \in X$.  % Let	$H^1(X, T_X)$ and  $\delta$ be as in equation  \ref{eq:LocGlob}. \\
%	Then the Anomaly Cancellation  Equation \ref{prop-anomalies-gen} in the form of Theorem \ref{prop:anomequm} must be  modified to:
%	\begin{eqnarray*}
	%	30 K_B^2+   \frac{1}{2}{  (\chi_{top}(X)-\sum_P m(P))}  &=	& 	H_{ch}  -	(\dim {\mathfrak g} - \operatorname{rk} {\mathfrak g} )
%	\end{eqnarray*}
	
%\end{definitiont}

As before we conclude:

\begin{corollary}[Bounds, III]\label{cor:boundsIII} Let $X \to B$ be an equidimensional elliptic fibration, with $X$ a Calabi-Yau threefold with $\mathbb{Q}$-factorial terminal singularities and $B$ possibly singular.  Assume that $\Sigma_m$ is disjoint from the divisorial components of $\Sigma$, that is $\Sigma_m \not \subset \Sigma_j,  \ \forall j$. 
	The quantities (1), (2) and $H_{unch}$ in (3), (5)  of the massless spectrum are bounded.
	
\end{corollary}

\begin{corollary}[Bounds, III']\label{cor:boundsIIIanom} Let $X \to B$ be an equidimensional elliptic fibration, with $X$ a Calabi-Yau threefold with $\mathbb{Q}$-factorial terminal singularities and $B$ possibly singular.  Assume that $\Sigma_m$ is disjoint from the divisorial components of $\Sigma$, that is $\Sigma_m \not \subset \Sigma_j,  \ \forall j$. 
The Anomaly Cancellation formula implies that  $H_{ch}$ is also bounded.
\end{corollary}

\subsection{The charged hypermultiples  $H_{ch}$}\label{sub:Hch}~

\smallskip

The  charged multiples in $H_{ch}$ have been computed explicitly in many cases in the physics literature. In \cite{GrassiWeigand2} we present some of the different methods for the computations and definitions in geometry.
 We review here the main points with the goal of establishing bounds for the geometric counterparts.
 
In the following 
	let   $X \to B$ be an  equidimensional elliptic Calabi-Yau threefold with $\mathbb{Q}$-factorial  terminal singularities and reduced discriminant $\Sigma$, with divisorial components $\Sigma_j$.
	Assume that $\Sigma_m \not \subset  \Sigma_j$, for every $j$. Let $g_j \stackrel{def}= g(\Sigma_j)$ be the geometric genus of $\Sigma_j$. 
	 %\begin{remark}
	If  $\mathfrak{g}(\Sigma_j)$ is not simply laced, let $\Sigma'_j$ be the curve parametrizing  the (curves spanning the) divisorial components of ${\mathfrak g}_j$.   Recall that  $\mathfrak g_j \stackrel{def}=\mathfrak g(\Sigma_j)$.
	%\end{remark}
%\begin{method}[Physics]\label{assignements}% Let   $X \to B$ be an  equidimensional elliptic Calabi-Yau threefold with $\mathbb{Q}$-factorial  terminal singularities and reduced discriminant $\Sigma$, with divisorial components $\Sigma_j$.
%	Assume that $\Sigma_m \not \subset  \Sigma_j$, for every $j$.   Let $g_j = g(\Sigma_j)$ be the geometric genus of $\Sigma_j$.
%	Let   $\operatorname{adj}_j  \stackrel{def}=\operatorname{adj} (\mathfrak{g}(\Sigma_j)$ be the adjoint representation of the Lie algebra ${\mathfrak g}_j $.
	
	%	\item Let $\rho$ be a representation of a Lie algebra
	%	$\mathfrak{g}$,
	%	with Cartan subalgebra $\mathfrak{h}$.
	%	The {\em charged dimension of $\rho$} is
	%	$(\dim \rho)_{ch}= \dim (\rho)- \dim (ker \rho|_{\mathfrak{h}})$.
%	To each irreducible component of the codimension one stratum $\Sigma_j$ one associates the  charged hypermultiplets representation
	%${\rm adj}_{ {\mathfrak g}(\Sigma_j)}$ 
%	$adj_j$ with  multiplicity $g(\Sigma_j)$.
%\end{method}

%The constructions in  physics imply the existence of the following  construction of hypermultiplets, summarized and substantiated  in \cite{GrassiWeigand2}.
In \cite{GrassiWeigand2} we also extend and reformulate the physics constructions and expectations  of the charged hypermultiplets:
\begin{hc}\label{ex:ph} \hfill
	
	\begin{enumerate}
		\item 	To each irreducible component %of the codimension one stratum 
		$\Sigma_j$ one associates the  charged hypermultiplets representation
		%${\rm adj}_{ {\mathfrak g}(\Sigma_j)}$ 
		$adj_j$ with  multiplicity $g(\Sigma_j)$, where   $\operatorname{adj}_j  \stackrel{def}=\operatorname{adj} (\mathfrak{g}(\Sigma_j))$ is the adjoint representation of the Lie algebra ${\mathfrak g}_j $.
		%	\item Let $\rho$ be a representation of a Lie algebra
		%	$\mathfrak{g}$,
		%	with Cartan subalgebra $\mathfrak{h}$.
		%	The {\em charged dimension of $\rho$} is
		%	$(\dim \rho)_{ch}= \dim (\rho)- \dim (ker \rho|_{\mathfrak{h}})$.
	%	To each irreducible component of the codimension one stratum $\Sigma_j$ one associates the  charged hypermultiplets representation
		%${\rm adj}_{ {\mathfrak g}(\Sigma_j)}$ 
	%	$adj_j$ with  multiplicity $g(\Sigma_j)$.
	%\item  If  $\mathfrak{g}(\Sigma_j)$ is not simply laced  there is a curve $\Sigma'_j$ parametrizing  the (curves spanning the) divisorial components of ${\mathfrak g}_j$. 
	\item    If  $\mathfrak{g}(\Sigma_j)$ is not simply laced, there exists a canonically determined representation $ \rho_0(\Sigma_j)$,  with multiplicity 
	$g(\Sigma'_j)- g(\Sigma_j)$.% \\ $g(\Sigma'_j)$ is the geometric genus of a curve $\Sigma'_j$ covering $\Sigma_j$.
	\item There exist (multi)-representations $\{\rho_Q\}_Q$ of $\oplus _j  {\mathfrak g}_j$,  where $Q \in \Sigma_j$ is  in the singular locus of $\Sigma$ %These are %called 
	(the matter representations).
%	\item  If  there is a section, there exist $\mathfrak{u}(1)_r$-charged hypermultiplets  $\{c_Q\}$,  where $Q  \in \Sigma_j$ is  in the singular locus of $\Sigma$,  with possibly ${\mathfrak g}(\Sigma_j)= \{e\}$. %These are in 1-1 correspondence with the %fibral 
%	holomorphic curves in the fibers of $X$ with vanishing intersection with the zero-section $S_0$.%, when it exists (the exceptional fibers of the Weierstrass model). 
	\item  If there is a section  and $Q   \in \Sigma_j$ is  in the singular locus of $\Sigma$,  with possibly ${\mathfrak g}(\Sigma_j)= \{e\}$, there exist $\mathfrak{u}(1)_r$-charged hypermultiplets  $\{c_Q\}$.
\item  If there is a multisection of index $m>1$,  and $Q    \in \Sigma_j$ is  in the singular locus of $\Sigma$, possibly ${\mathfrak g}(\Sigma_j)= \{e\}$,  there exists  hypermultiplets with multiplicity $c_Q$,  charged under the discrete \footnote{See \cite{Lee:2022swr} and references therein for an overview.} gauge group which is associated with the index of the multisection.
\end{enumerate}
\end{hc}

\begin{remark} If the general  singular fibers of the elliptic fibration  are  irreducible, then $\{Q\}$ %and $Q$ are 
	are the loci in the discriminant where the reducible fibers are located. $c_Q$ %and $c_Q$  are 
	is expected to be generically $1$, but it is not always the case, see Example \ref{ex:NR}. 
Physics dualities relate the quantities $\{c_Q\}$ %and $c_Q'$
 in (4) and (5) to the BPS states, and  the (relative) Gopakumar-Vafa invariants at genus zero.  
%The quantities $c_Q$ are in 1-1 correspondence with the %fibral 
%holomorphic curves in the fibers of $X$ with vanishing intersection with the zero-section $S_0$. %, when it exists (the exceptional fibers of the Weierstrass model). 
The recent paper \cite{Knapp:2021vkm} computes the (relative) genus zero Gopakumar-Vafa invariants  for general elliptic fibration  with a multisection of index $5$.  See also \cite{Lee:2022swr} and references therein for the  bounds related to the discrete charges.
We elaborate on (4) and (5) in \cite{GrassiWeigand2}.  
\end{remark}

\begin{definition} \label{def:chargedim}
 Let $\rho$ be a representation of a Lie algebra
		$\mathfrak{g}$,
		with Cartan subalgebra $\mathfrak{h}$.
		The {\em charged dimension of $\rho$} is
		$(\dim \rho)_{ch}= \dim (\rho)- \dim (ker \rho|_{\mathfrak{h}})$.
\end{definition}
\begin{expectation} 
%	\begin{enumerate}	\item 
	$H_{ch}$ is a combination of the charged dimensions of the quantities  in  \ref{ex:ph}.  %the Gravitational Anomaly cancellation  with 
	%\begin{equation*} \label{HminusV}
	%	H - V + 29 T  = 273, 
	%\end{equation*}
	%is satisfied. Or equivalently, =
%\item The   counterpart of the Gravitational Anomaly cancellation 
	%\begin{eqnarray*}
	%	30 K_B^2+   \frac{1}{2}{  (\chi_{top}(X)-\sum_P m(P))} +  (\dim {\mathfrak g} - \operatorname{rk} {\mathfrak g} )  &=	& 	H_{ch} 
%	\end{eqnarray*}
	%is satisfied.
%\end{enumerate}

\end{expectation}
 
Although it is an open question how to compute explicitly the representations and their combinations  in \ref{ex:ph},  they have been computed under general conditions %and in  various hypothesis 
and in many examples.  We present some of these results in the Appendix \ref{subsec:examples}). Accordingly, in  \cite{GrassiWeigand2} we make the following:   %It is an open question how to compute explicitly the representations and their combinations  in \ref{ex:ph}, but they have been computed under general conditions and in % various hypothesis  
%and in many examples. We present some in the Appendix \ref{subsec:examples}).

%	$H_{ch}$   in (4)
%are in 1-1 correspondence with the %fibral 
%holomorphic curves in the fiber of $Y$ with vanishing intersection with the zero-section $S_0$ (the exceptional fibers of the Weierstrass model).  $H_{ch}$ can be  computed by %either
%their Gopakumar-Vafa invariants at genus zero.  %or 
%the localised deformations of the singular fibration $\bar Y \to B$.
%They are generically expected to be $1$, but see  \cite{Grassi:2021wii}.

% However the $H_{ch}$ have been computed explicitly in many cases. See \cite{GrassiWeigand2} we present different constructions.
%\end{remark}

\begin{conjecture}\label{conjHch}
	Let   $X \to B$ be an  equidimensional elliptic Calabi-Yau threefold with $\mathbb{Q}$-factorial  terminal singularities and reduced discriminant $\Sigma$, with divisorial components $\Sigma_j$.
	Assume that $\Sigma_m \not \subset  \Sigma_j$, for every $j$.  Then the charged multiplets 	$H_{ch}$  in the massless spectrum are:
	%\begin{enumerate}
	%	\item 
%	$H_{ch}$ is as follows:
	\begin{equation*}
		\boxed{	H_{ch} = \sum_jg (\Sigma_j) (\dim \operatorname{adj}_j) _{ch} + \sum_j(g_j' -g_j)  (\dim \rho_{0,j})_{ch} 
			+ \sum_Q    (\dim \rho_Q)_{ch} +    
			\sum_{Q} c_{Q} % +    	\sum_{Q'} c_{Q'}
	 }
	\end{equation*}	
	
	with the quantities defined as in \ref{def:chargedim} and \ref{ex:ph}.
\end{conjecture}
%	\item 
%Then
\begin{proposition}\label{prop:geomanomhc}
	Let   $X \to B$ be an  equidimensional elliptic Calabi-Yau threefold with $\mathbb{Q}$-factorial  terminal singularities and reduced discriminant $\Sigma$, with divisorial components $\Sigma_j$.
Assume that $\Sigma_m \not \subset  \Sigma_j$, for every $j$ and that \ref{conjHch} holds. 

The   Geometric Anomaly Equation  then becomes
	\begin{equation*}%\label{equ:GeomAnH}
\boxed{
	\begin{array}{rcl}
		 30K_B^2  +  \frac{1}{2}{  (\chi_{top}(X)-\sum_P m(P))} &+ & (\dim {\mathfrak g} - \operatorname{rk} {\mathfrak g} ) = \\
= \sum_j(g (\Sigma_j) (\dim \operatorname{adj}_j) _{ch} &+& \sum_j(g_j' -g_j)  (\dim \rho_{0,j})_{ch} 
	+ \sum_Q    (\dim \rho_Q)_{ch} +    
	\sum_{Q} c_{Q} % +    \sum_{Q'} c_{Q'}.
\end{array}
}
		\end{equation*}
\end{proposition}

%The Geometric Anomaly Equation in  the boxed formula of Proposition \ref{prop:geomanomhc} has been  verified  under general conditions and in % various hypothesis many cases (Appendix \ref{subsec:examples}).

%\end{conjecture}
%\end{enumerate}	
%\begin{conjecture}Let   $X \to B$ be an  equidimensional elliptic Calabi-Yau threefold with $\mathbb{Q}$-factorial  terminal singularities and reduced discriminant $\Sigma$, with divisorial components $\Sigma_j$.
%	Assume that $\Sigma_m \not \subset  \Sigma_j$, for every $j$ and $H_{ch}$ as in Conjecture \ref{conjHch}. Then
%	the Geometric Anomaly Equation 	\begin{eqnarray*}
	%	30 K_B^2+   \frac{1}{2}{  (\chi_{top}(X)-\sum_P m(P))} +  (\dim {\mathfrak g} - \operatorname{rk} {\mathfrak g} )  &=	& 	H_{ch} 
%	\end{eqnarray*}
%	is satisfied.
	
%\end{conjecture}
%The conjecture has been verified under various hypothesis and in many cases.

% \begin{expectation} With the same hypothesis as in Assignment \ref{assignements} the Gravitational Anomaly Cancellation
	%	 \begin{equation*} \label{HminusV}
		%	H - V + 29 T  = 273, 
		%	\end{equation*}
	%	is satisfied.
	%\end{expectation}
	
	\medskip

\section{Geometric Anomaly and other bounds in geometry}

The Geometric Anomaly Equation in  the boxed formula of Proposition \ref{prop:geomanomhc} \   is verified  under general conditions and in % various hypothesis 
many cases; we present a few in Appendix \ref{subsec:examples}.

\begin{corollary}[Bounds,  IV]\label{cor:boundsIV} Let   $X \to B$ be an  equidimensional elliptic Calabi-Yau threefold with $\mathbb{Q}$-factorial  terminal singularities and reduced discriminant $\Sigma$, with divisorial components $\Sigma_j$.
	Assume that $\Sigma_m \not \subset  \Sigma_j$, for every $j$. 

 Conjecture \ref{conjHch} and the Geometric Anomaly Cancellation  in Proposition \ref{prop:geomanomhc} imply that 
 charged hypermultiplets $(\rho_{0,j})_{ch} $, $(\rho_Q)_{ch}$ %$\{c_{Q'} \} $ 
 and  $\{c_{Q}\} $,  the genus zero  relative Gopakumar-Vafa invariants in  the boxed formula  \ref{conjHch}, the divisorial components $\Sigma_j$ and the singular points of the discriminants are bounded.
\end{corollary}
In particular the matter representations are bounded.
\begin{proof}
The quantities on right hand side in the boxed formula  of Conjecture \ref{conjHch} and Proposition \ref{prop:geomanomhc} are non-negative and Corollary \ref{cor:boundsIIIanom} implies the statement.
\end{proof} 

In  \cite{GrassiWeigand2} we show that the Geometric Anomaly Equation  implies that ${\mathfrak g}$ and  the charged hypermultiplets $(\rho_{0,j})_{ch} $, $(\rho_Q)_{ch}$ and  $\{c_{Q}, c'_{Q}\} $ (relative genus zero Gopakumar-Vafa invariants) in  the boxed formula  \ref{conjHch} are birational invariants of the relatively minimal elliptic fibration. These invariants which we call \textit{stringy Kodaira}  are a natural generalization to higher dimension  of Kodaira's classification of singular fibers of elliptic surfaces. 

\begin{corollary}[Bounds,  V]\label{cor:boundsV}
There are finitely many  types of stringy Kodaira fibers for elliptically fibered Calabi-Yau threefolds.
\end{corollary}

\subsection{Multiple fibers, indices  of multisections and bounds}~
\smallskip

The physics of elliptic fibrations with multiple fibers is still not well understood. % (but see for example \cite{Anderson:2017aux},  \cite{Anderson:2019kmx}, \cite{Kohl:2021rxy},
%Let $Q \in \Sigma_m$ but $ Q \not \in \Sigma_j$, $\forall j$; 
The multiple fibers  over the points  $Q \in \Sigma_m$ but $ Q \not \in \Sigma_j$, $\forall j$ do not contribute to the  massless spectrum.  %in \ref{thm:SpectrumP}. 
 	Their multiplicity is however associated  to \textit{discrete symmetries} (see  for example \cite{Lee:2022swr} and reference wherein).  %The discrete symmetries are bounded because so are the multisections.
The multiplicity of the fibers, the  associated {discrete symmetries}, and the index of the multi-sections are related.

It is expected that an elliptic Calabi-Yau  $X \to B$ without sections can be transformed,   with a generalization of the conifold transition, to a different elliptic  Calabi-Yau $Y \to B'$ with section, $\mathbb{Q}$-factorial terminal  singularities, in the philosophy of Reid's \cite{ReidCan}, and that   the index of the multisection is bounded by $\operatorname{rk} MW(Y/B'$).  There is an  explicit prediction for the bounds, based on different reasoning (Conjecture \ref{conj:multIndex}).
The results of Gross \cite{Gross1994}, Filipazzi-Hacon-Svaldi \cite{Filipazzi:2021dcw} imply that  the index of the multisections and  $\operatorname{rk} MW$ are bounded.
%  for elliptically fibered Calabi-Yau, \cite{Gross1994} and c. %As a consequence so are the associated discrete symmetries. 

By contrast, \cite{Filipazzi:2021dcw} shows that the index of the multisection of elliptically fibered threefolds of Kodaira dimension $2$  is not necessarily  bounded.
We conjecture that boundedness for Calabi-Yau and Reid's  philosophy of \cite{ReidCan} are related.

%The possible fibers for elliptic 
% Bounds on torsion  Mordell-Weil  \cite{ShimadaK3Elliptic}.

\section{Explicit bounds}\label{subsec:ExplBounds}~
%\begin{corollary}
%	The discrete symmetries are bounded.
%\end{corollary} 
%The quest for explicit bounds is open.
It is an important question  in string theory  to find explicit bounds of the spectrum, also in the context of the ``Swampland Program"   \cite{Vafa:2005ui}:
\begin{proposition}[\cite{Taylor:2012dr}]\label{prop:taylorh21} Let $X \to B$ be an elliptic fibration with section, $X$ a Calabi-Yau manifold  and $B$ rational, then $\mathrm{CxDef} (X) \leq 491$.
\end{proposition}
\begin{proof}  The starting point of  argument is that $X \to B$ is birational to an elliptic fibration $X \to \bar B$, where $B$ is either $\mathbb{P}^2$ or $\mathbb{F}_n$ a Hirzebruch surface with $n\leq 12$ \cite{Grassi1991}. Then  an explicit computation gives that the general Jacobian fibration over such bases $B$ have dimension of complex deformation $\leq 491$. Note also the $\mathrm{CxDef}(X)$ is a birational invariant of the minimal model of $X$.
\end{proof}
\begin{corollary}
	 Let   $\pi: X \to B$ be an  equidimensional %be an equidimensioanl 
	 elliptic fibration without a section, $X$ a  Calabi-Yau threefold,  %without a section, with $X$,
	  $B$ smooth and rational. Assume that the  fibration is general, that is the  fibers over the general points of $\Sigma$ are nodal (type $I_1$), %the fibration is 
	  the singular locus of $\Sigma$  consists of  nodal points $Q$ and the fibers  $\pi ^{-1}(Q)$ are of type   $I_2$. Then $\mathrm{CxDef}(X)\leq 491$.
\end{corollary}
\begin{proof}  %general the non-abelian gauge algebra is trivial,
	The Jacobian fibration $\pi^J: Jac (X)  \to B$ has fibers of  type $I_1$ over the points $Q$ and $\mathbb{Q}$-factorial terminal ordinary double point singularities at the nodal point of  each fiber  $({\pi ^J}) ^{-1}(Q)$.  By comparing the topological Euler characteristics  and the complex deformations as in  \ref{def:specSing} we  obtain $\mathrm{CxDef}(X)=\mathrm{CxDef} Jac(X)-\sum_{Q} 1$.  $J(X)$ can be deformed to a smooth elliptically fibered Weierstrass model %\cite{NamikawaSteenbrink} and \cite{KollarElliptic2015} 
	over $B$ and we then conclude by Proposition \ref{prop:taylorh21}.
\end{proof}

\begin{corollary}
Let $X \to B$ be an elliptic fibration, $X$ Calabi-Yau. Assume that the mirror of $X$ is an elliptic fibration with section. Then $h^{1,1}(X) \leq 491$.
\end{corollary}

%It is conjectured that such elliptic fibrations have a mirror \cite{Schimannek:2021pau, Oehlmann:2016wsb, Klevers:2014bqa}.

\begin{conjecture}[Index of a fibration, Physics]\label{conj:multIndex}
	Let  $X \to B$ be an elliptic fibration, $X$, Calabi-Yau, then $n \leq 6$, where $n$ is the index of a multi-section.
\end{conjecture}
The  highest known  multisection index for a Calabi-Yau threefold is  $ n = 5$%,  studied in
 \cite{Knapp:2021vkm}.

The conjecture is derived from a duality argument based on the following:
\begin{result}[\cite{Lee:2022swr} and references therein]
Let  $X \to B$ be an elliptic fibration between manifolds with section, $X$, Calabi-Yau, then $\operatorname{tor}\operatorname{MW}(X/B)= \mathbb{Z}/{n_1\mathbb{Z}}\times \mathbb{Z}/{n_2 \mathbb{Z}}$, $(n_1, n_2)= (2,2), (3,3), (2,4)$ or $\operatorname{tor}\operatorname{MW}(X/B)= \mathbb{Z}/{n\mathbb{Z}}$, $ n \leq 6$.

\end{result}
%There are also results about the bounds on the torsion part of the Mordell-Weil group of   $X \to B$ and $X$. Calabi-Yau , see \cite{Lee:2022swr} for recent results and references therein. 
These extend the work of Miranda, Persson and Shimada on K3 surfaces  \cite{ShimadaK3Elliptic}. 

\begin{result}[\cite{Lee:2019skh}]
Let  $X \to B$ be an elliptic fibration between manifolds, $X$, Calabi-Yau   then $ {\rm rk}(\operatorname{MW}(X)) \leq 20$ if  $B \neq \mathbb P^2$ and $ {\rm rk}(\operatorname{MW}(X)) \leq 24$ if $B= \mathbb P^2$ 
\end{result}

The above bounds are not believed not to be sharp.  The  highest known  Mordell-Weil rank for a Calabi-Yau threefold is  $ {\rm rk}(\operatorname{MW}(X)) = 10$ \cite{Grassi:2021wii}; see  also %the references wherein for examples with Mordell-Weil ranks up to $9$, and also
 \cite{Elkies}.

\medskip

\noindent{\bf  Acknowledgments: } We thank S. Filipazzi, R. Svaldi, T. Weigand for discussions as well as C. Hacon,   P. Oehlmann, W. Taylor,  A. Zanardini  and the anonymous referee for   comments  on an earlier draft. 
This  work  is partially supported by PRIN2017 ``Moduli and Lie Theory" and  by  GNSAGA of INdAM.
We also thank the organizers of the  2022  JAMI Conference in honor of V.V.  Shokurov and  of  the workshop  on Birational Geometry at  the Mathematisches Forschungsinstitut in Oberwolfach. Last but not least, thank you to  Prof.  Shokurov.% for stimulating questions.

\bigskip

	\section{Appendix: Examples}\label{subsec:examples}
	The Anomaly Cancellation and the Geometric Anomaly Equation have been proven for the many examples and the general conditions which have motivated our constructions.  Here we  describe a few.
\begin{ex}
	Let $X \to B$ an equidimensional elliptic fibration, with $X$  a smooth Calabi-Yau threefold and $B$ a  smooth Enriques surface.  % Let	$H^1(X, T_X)$ and  $\delta$ be as in equation  \ref{eq:LocGlob}. \\
	Then Conjecture \ref{conjHch} and the Anomaly Cancellation  Equation  in \ref{prop-anomalies-gen} and Proposition \ref{prop:geomanomhc}
	%\begin{eqnarray*}
	%	30 K_B^2+   \frac{1}{2}{  (\chi_{top}(X))}  &=	& 	H_{ch}  -	(\dim {\mathfrak g} - \operatorname{rk} {\mathfrak g} )
	%	\end{eqnarray*}
always hold.

\end{ex}
\begin{proof}
Since the  fibration is isotrivial $ {\mathfrak g} = \{e\}$, $\chi_{top}(X)=	H_{ch} =0$.
\end{proof}

\begin{ex}(\cite[Theorem 2.2]{GrassiMorrison03}, following  \cite{Sethi:1996es}.) Let $X \to B$ an equidimensional elliptic fibration, with $X$  a smooth Calabi-Yau Weierstrass threefold and $B$   smooth.  Then the Anomaly  Equation  in  Proposition \ref{prop-anomalies-gen}
%\begin{eqnarray*}
%	30 K_B^2+   \frac{1}{2}{  (\chi_{top}(X))}  &=	& 	H_{ch}  -	(\dim {\mathfrak g} - \operatorname{rk} {\mathfrak g} )
%	\end{eqnarray*}
%always 
holds.
\begin{proof}
Any  smooth Calabi-Yau Weierstrass model $X=W$ satisfies the geometric anomaly,  because   $	H_{ch}  -	(\dim {\mathfrak g} - \operatorname{rk} {\mathfrak g} )=0$ and  $\chi_{top}(X)= 60 K_{B}^2$.
\end{proof}
\end{ex}
\begin{definition}
$X\to B$ is a general  elliptic fibration  if it is relatively minimal, equidimensional (as in Theorem \ref{thm:birminmodel})  and 
the discriminant is of the form $\Sigma = \Sigma_1 \cup \Sigma_0$, where  $\Sigma_1$ is a smooth curve, $\Sigma_0$ denotes the locus where the general fiber is a nodal elliptic curve.
%and $\operatorname{rk} MW(X)=0$, that is ${\mathfrak g}= {\mathfrak g}(\Sigma_1$). 
\end{definition}
\begin{ex}[General elliptic fibration with 5-section \cite{Knapp:2021vkm}]
 Let 	$X\to \mathbb{P}^2$ be a general elliptic fibration with a 5-section,  and $X$ Calabi-Yau. Assume that $%{\mathfrak g}= {\mathfrak g}(\
 \Sigma_1=\emptyset$.   Then  Conjecture \ref{conjHch} holds and the Anomaly Cancellation  Equation  in \ref{prop-anomalies-gen} and  in Proposition \ref{prop:geomanomhc} hold as well.  
\end{ex}

\begin{thm}[General elliptic fibration with section, \cite{GrassiWeigand2}\cite{GrassiMorrison03}]\label{thm:wgen} Let 	$X\to B$ be a general elliptic fibration with section,  and $X$ Calabi-Yau. Assume that $\operatorname{rk} MW(X)=0$.   Then  Conjecture \ref{conjHch} holds with 
%\begin{enumerate}
%	\item The representations
	%\r  &=
	%30 K_B^2 + \frac{1}{2}( \chi_{top}(X) +  \sum_P  m(P) &=
	%& (g-1) (\dim \operatorname{adj}) _{ch} + (g' -g)  (\dim \rho_0)_{ch}  + \sum_Q    (\dim \rho_Q)_{ch} +    \sum_P \tau(P) \,, 
	%\end{eqnarray}
	%where
	%$\r$ as in Definition \ref{defr}; let 
	%$Q$ is a  singular point of $\Sigma_1$, that is a non-general singularity of $\Sigma$.
	%$\rho_Q$ is the 
	%associated localized representation 
	%obtained as in 
	%(Section \ref{sec_localized}).
	%, together with a multiplicity
	%$g\stackrel{def}=g(\Sigma_1)$ and $g'\stackrel{def}=g(\Sigma'_1)$ denote the genus of 
	%the discriminant component
	%$\Sigma_1$ and, respectively, of its finite branched cover $\Sigma_1'$ 
	%occurring in 
	%(Lemma \ref{def:g'}.) 
	%${\rm adj} = {\rm adj}_{\mathfrak{g}(\Sigma_1)}$, $\rho_0 = \rho_0(\Sigma_1)$ are the unlocalized representations 
	%according to 
	%(Lemma \ref{method-unloc}). $P$  is a singular points of $X$ with Tyurina number $\tau(P)$. 
	$\rho_0$,  $\rho_Q$  and the Milnor numbers 
	are as in  Table \ref{tab:A}.

\begin{table}[ht!]
	
	\begin{center}
		\scalebox{0.7}{
			\begin{tabular}{|| c|c||c|c|c||c|c|c|c ||c|c||}\hline
				Type & $ \mathfrak{g}$ &  $\rho_0$ &  $\rho_{Q^\ell _1}$ & $\rho_{Q^\ell_2}$ & $(\dim \operatorname{adj}) _{ch} $&${(\dim  {\rho_0})}_{ch}$&$\dim {( \rho_{Q^\ell_1})}_{ch}$&${\dim {( \rho_{Q^\ell_2})}_{ch}}$ & $m(P_1)$ & $m(P_2)$\\ \hline
				 $I_{1}$ & $\{e\} $    &    &-- &-- & $0$ &$0$&$0$  &  0  & 0& 1\\ \hline
			 	$I_{2}$ & $\operatorname{su} (2)$   &    &--&$\operatorname{fund}$ & $2$ &$0$&  $0$  & $2 $&  &  \\ \hline
			 	$III$ & $\operatorname{su} (2)$ & &   $2\times\operatorname{fund}$&  & $2$ &$0$&$4$  & &  &    \\ \hline  
			$I_{3}$ & $\operatorname{su} (3) $      &  &-- &$\operatorname{fund}$  & $6$ &$0$&  $ 0$  & $3$  &  &   \\ \hline
			$I_{2k} $, $k\ge 2$ & $\mathfrak{sp} (k) $    &$\Lambda^2_0$&--&
				$\operatorname{fund}$ &$2k^2$  &$2k^2-2k$&$0$&$2k $  &  &    \\ \hline
$I_{2k+1}$, $k\ge1$ & $\mathfrak{sp} (k)$     &$\Lambda^2+2\times\operatorname{fund}$ &$\frac12\operatorname{fund}$  & $\operatorname{fund}$ &$2k^2$  &$2k^2+2k$&$k$ &  $2k$ & 1 & 0 \\ \hline
	$I_{n}$, $n\ge4$ & $\operatorname{su} (n) $ && $\Lambda^2$&  $\operatorname{fund}$ & $n^2 -n$ &$0$& $\frac12(n^2-n) $  & $n $&  &    \\ \hline
	$IV$ & $\mathfrak{sp}(1)$ & $\Lambda^2+2\times\operatorname{fund}$&$\frac12\operatorname{fund}$   $2$ &$4$ & $1$  &  & & &  &    \\ \hline
		$IV$ & $\operatorname{su} (3)$ && $3\times\operatorname{fund}$& & $6$ &$0$&  $9 $  &&  & 
	\\ \hline
	$I_0^*$ & $\mathfrak g_2$ &$\mathbf{7}$&-- &  &  $12$ &$6$&$0$ &   &  &   \\ \hline
	$I^*_{0}$ & $ \mathfrak{so}(7) $ & $
$&-- &$\operatorname{spin}$  &  $18$ &$6 $&$0$ & $8 $  &  &  \\ \hline	
$I^*_{0} $ & $\mathfrak{so} (8) $ &&  ${\operatorname{vect}}$&$\operatorname{spin}_\pm$ &  $24$ &$0$&$8$ & $8$ &  &     \\ \hline
$I^*_{1}$ & $\mathfrak{so} (9) $ &${\operatorname{vect}}$&-- &$\operatorname{spin}$ &  $32$ &$8$&$0$  &$16$ &  &   \\ \hline	
$I^*_{1} $ & $\mathfrak{so} (10) $ && ${\operatorname{vect}}$ & $\operatorname{spin}_\pm$&  $40$ &$0$&$10$ & $16$   &  &    \\  \hline
$I^*_{2}$ & $\mathfrak{so} (11) $&${\operatorname{vect}}$&-- &$\frac12\operatorname{spin}$  &  $50$ &$10$&$0$  &$16 $ &  &   \\ \hline
	$I^*_{2} $ & $\mathfrak{so} (12) $ &&  ${\operatorname{vect}}$& $\frac12\operatorname{spin}_\pm $ &  $60$ &$0$& $12$& $16$  &  & \\ \hline
$I^*_{n}$, $n\ge3$ &${\mathfrak {so}} (2n+7)$&${\operatorname{vect}}$&-- &-- &  $2(n{+}3)^2$ &$2n{+}6$&$0$  & --&  &  \\
\hline
$I^*_{n} $, $n\ge3$ & ${\mathfrak {so}} (2n+8) $&&  ${\operatorname{vect}}$&--&  $2(n{+}3)(n{+}4)$ &$0$& $2n{+}8 $   & --  &  &    \\ \hline
$IV ^*$ & $\mathfrak f_4$ &$\mathbf{26}$&--&  &  $48$ &$24$&$0$   & &  & \\ \hline
$IV ^*$ & $\mathfrak e_6$ &&  $\mathbf{27}$ &  & $72$&$0$& $27 $  &  &  &  \\ \hline
$III ^*$ & $\mathfrak e_7$ & $\frac12\mathbf{56}$ &  &  $126$ &$0$& $28 $      &  & &  & \\ \hline
$II ^*$ & $\mathfrak e_8$ & &--  & & $240$ &$0$&  --  & &  &    \\ \hline
	\end{tabular}
	}
	\end{center}
\medskip
\caption{ \small{
%The representations which occur for a general fibration
%	hypotheses. The Milnor numbers are  listed in the last two columns
%	The associated representations and the Milnor numbers  are independent of the particular relative minimal model  %by  Theorem  \ref{Correps}. %Cases with non-minimal Weierstrass model are denoted ``\nonmin''. 
}}\label{tab:A}
\end{table}
The Anomaly Cancellation  Equation  in \ref{prop-anomalies-gen} and  in Proposition \ref{prop:geomanomhc} hold as well.  % Furthermore:
\end{thm}

\begin{proof} 
Here  ${\mathfrak g}= {\mathfrak g}(\Sigma_1$).
\end{proof}
%\newpage
\begin{ex}\cite{GRTW}   Let 	$X\to B$ be a general elliptic fibration with section,  and $X$ Calabi-Yau. Assume that $\operatorname{rk} MW(X)=1$.   Then  Conjecture \ref{conjHch} holds. The table with the representations and the hypermultiplets $H_{ch}$, which extends Table A, includes a column with  $\{c_Q=1\}$.
	 %In \cite{GRTW} the Theorem \ref{thm:wgen} \  is {extended to} $\operatorname{rk} MW(X) =1$, that is ${{\mathfrak g}={\mathfrak g}(\Sigma_1) \oplus \mathfrak{u}(1)}$. The corresponding Table  includes $\{c_Q\}$.
\end{ex}
%\begin{ex}\cite{Grassi:2021wii}
%The Namikawa-Rossi Schoen
%\end{ex}

%\begin{ex}\cite{GRTW}
%$\operatorname{rk}(MW)=2$ under general hypothesis
%\end{ex}

\begin{ex}\label{ex:NR}\cite{Grassi:2021wii}   Let $X \to B$  be the   Schoen-Namikawa-Rossi  Calabi-Yau  threefold,  with $\operatorname{rk} MW(X/B)=10$.  Then  Proposition \ref{prop:geomanomhc} holds with $\mathfrak{g}=\mathfrak{u}(1)^{10}$  and  $\{c_Q>1\}$.

\end{ex}


\begin{thebibliography}{10}
	
	\bibitem{Anderson:2018heq}
	{\sc L.~B. Anderson, A.~Grassi, J.~Gray, and P.-K. Oehlmann}, {\em F-theory on
		quotient threefolds with {$(2,0)$} discrete superconformal matter}, J. High
	Energy Phys.,  (2018), pp.~098, front matter+76.
	
	\bibitem{Anderson:2019kmx}
	{\sc L.~B. Anderson, J.~Gray, and P.-K. Oehlmann}, {\em {F-Theory on Quotients
			of Elliptic Calabi-Yau Threefolds}}, JHEP, 12 (2019), p.~131.
	
	\bibitem{ArrasGrassiWeigand}
	{\sc P.~Arras, A.~Grassi, and T.~Weigand}, {\em Terminal singularities,
		{M}ilnor numbers, and matter in {F}-theory}, J. Geom. Phys., 123 (2018),
	pp.~71--97.
	
	\bibitem{BraunMorrison}
	{\sc V.~Braun and D.~R. Morrison}, {\em F-theory on genus-one fibrations}, J.
	High Energy Phys.,  (2014), p.~132.
	
	\bibitem{Elkies}
	{\sc N.~D. Elkies}, {\em K3 surfaces and elliptic fibrations in number theory}.
	\newblock Banff Workshop 18w5190, 2018.
	
	\bibitem{Filipazzi:2021dcw}
	{\sc S.~Filipazzi, C.~D. Hacon, and R.~Svaldi}, {\em Boundedness of elliptic
		{C}alabi-{Y}au threefolds}.
	\newblock {ArXiv: 2112.01352}, 2021.
	
	\bibitem{Grassi1991}
	{\sc A.~Grassi}, {\em On minimal models of elliptic threefolds}, Math. Ann.,
	290 (1991), pp.~287--301.
	
	\bibitem{GrassiSingBase1993}
	\leavevmode\vrule height 2pt depth -1.6pt width 23pt, {\em The singularities of
		the parameter surface of a minimal elliptic threefold}, Internat. J. Math., 4
	(1993), pp.~203--230.
	
	\bibitem{GrassiEqui}
	{\sc A.~Grassi}, {\em Log contractions and equidimensional models of elliptic
		threefolds}, J. Algebraic Geom., 4 (1995), pp.~255--276.
	
	\bibitem{GrassiMorrison03}
	{\sc A.~Grassi and D.~R. Morrison}, {\em Group representations and the {E}uler
		characteristic of elliptically fibered {C}alabi-{Y}au threefolds}, J.
	Algebraic Geom., 12 (2003), pp.~321--356.
	
	\bibitem{GrassiMorrison11}
	{\sc A.~Grassi and D.~R. Morrison}, {\em Anomalies and the {E}uler
		characteristic of elliptic {C}alabi-{Y}au threefolds}, Commun. Number Theory
	Phys., 6 (2012), pp.~51--127.
	
	\bibitem{GRTW}
	{\sc A.~Grassi, N.~Raghuram, A.~Turner, and T.~Weigand}, {\em Work in
		progress}.
	
	\bibitem{GrassiWeigand2}
	{\sc A.~Grassi and T.~Weigand}, {\em On topological invariants of algebraic
		threefolds with ({$\mathbb Q$}-factorial) singularities}.
	\newblock {ArXiv: 1804.02424 [math.AG]} (under revision), 2018.
	
	\bibitem{Grassi:2021wii}
	{\sc A.~Grassi and T.~Weigand}, {\em {Elliptic threefolds with high
			Mordell\textendash{}Weil rank}}, Commun. Num. Theor. Phys., 16 (2022),
	pp.~733--759.
	
	\bibitem{Green:1984bx}
	{\sc M.~B. Green, J.~H. Schwarz, and P.~C. West}, {\em {Anomaly Free Chiral
			Theories in Six-Dimensions}}, Nucl. Phys. B, 254 (1985), pp.~327--348.
	
	\bibitem{Gross1994}
	{\sc M.~Gross}, {\em A finiteness theorem for elliptic {C}alabi-{Y}au
		threefolds}, Duke Math. J., 74 (1994), pp.~271--299.
	
	\bibitem{Knapp:2021vkm}
	{\sc J.~Knapp, E.~Scheidegger, and T.~Schimannek}, {\em {On genus one fibered
			Calabi-Yau threefolds with 5-sections}}.
	\newblock {ArXiv: 2107.05647 [hep-th]}.
	
	\bibitem{Kohl:2021rxy}
	{\sc F.~B. Kohl, M.~Larfors, and P.-K. Oehlmann}, {\em {F-theory on 6D
			symmetric toroidal orbifolds}}, JHEP, 05 (2022), p.~064.
	
	\bibitem{Lee:2022swr}
	{\sc S.-J. Lee and P.-K. Oehlmann}, {\em {Geometric Bounds on the 1-Form Gauge
			Sector}}.
	\newblock {ArXiv: 2212.11915}, 2022.
	
	\bibitem{Lee:2019skh}
	{\sc S.-J. Lee and T.~Weigand}, {\em {Swampland Bounds on the Abelian Gauge
			Sector}}, Phys. Rev. D, 100 (2019), p.~026015.
	
	\bibitem{Morrison:1996na}
	{\sc D.~R. Morrison and C.~Vafa}, {\em {Compactifications of F theory on
			Calabi-Yau threefolds. 1}}, Nucl. Phys. B, 473 (1996), pp.~74--92.
	
	\bibitem{Morrison:1996pp}
	\leavevmode\vrule height 2pt depth -1.6pt width 23pt, {\em {Compactifications
			of F theory on Calabi-Yau threefolds. 2.}}, Nucl. Phys. B, 476 (1996),
	pp.~437--469.
	
	\bibitem{NamikawaSteenbrink}
	{\sc Y.~Namikawa and J.~H.~M. Steenbrink}, {\em Global smoothing of
		{C}alabi-{Y}au threefolds}, Invent. Math., 122 (1995), pp.~403--419.
	
	\bibitem{Oehlmann:2019ohh}
	{\sc P.-K. Oehlmann and T.~Schimannek}, {\em {GV-Spectroscopy for F-theory on
			genus-one fibrations}}, JHEP, 09 (2020), p.~066.
	
	\bibitem{PS09}
	{\sc Y.~G. Prokhorov and V.~V. Shokurov}, {\em Towards the second main theorem
		on complements}, J. Algebraic Geom., 18 (2009), pp.~151--199.
	
	\bibitem{ReidCan}
	{\sc M.~Reid}, {\em Canonical {$3$}-folds}, in Journ\'{e}es de {G}\'{e}ometrie
	{A}lg\'{e}brique d'{A}ngers, {J}uillet 1979/{A}lgebraic {G}eometry, {A}ngers,
	1979, Sijthoff \& Noordhoff, Alphen aan den Rijn---Germantown, Md., 1980,
	pp.~273--310.
	
	\bibitem{Sethi:1996es}
	{\sc S.~Sethi, C.~Vafa, and E.~Witten}, {\em {Constraints on low dimensional
			string compactifications}}, Nucl. Phys. B, 480 (1996), pp.~213--224.
	
	\bibitem{ShimadaK3Elliptic}
	{\sc I.~Shimada}, {\em On elliptic {$K3$} surfaces}, Michigan Math. J., 47
	(2000), pp.~423--446.
	
	\bibitem{Taylor:2012dr}
	{\sc W.~Taylor}, {\em {On the Hodge structure of elliptically fibered
			Calabi-Yau threefolds}}, JHEP, 08 (2012), p.~032.
	
	\bibitem{Vafa:1996xn}
	{\sc C.~Vafa}, {\em {Evidence for F theory}}, Nucl. Phys. B, 469 (1996),
	pp.~403--418.
	
	\bibitem{Vafa:2005ui}
	\leavevmode\vrule height 2pt depth -1.6pt width 23pt, {\em {The String
			landscape and the swampland}},  (2005).
	
	\bibitem{Weigand:2018rez}
	{\sc T.~Weigand}, {\em {F-theory}}, PoS, TASI2017 (2018), p.~016.
	
\end{thebibliography}
\end{document}